\documentclass{amsart}
\usepackage{amssymb}
\usepackage{relsize}
\usepackage{enumerate}
\allowdisplaybreaks

\usepackage{comment} 
\usepackage{soul}
\usepackage{float}
\let\oldtocsection=\tocsection 
\let\oldtocsubsection=\tocsubsection 
\renewcommand{\tocsection}[2]{\hspace{0mm}\oldtocsection{#1}{#2}}
\renewcommand{\tocsubsection}[2]{\hspace{2em}\oldtocsubsection{#1}{#2}}

\usepackage{hyperref}
\hypersetup{
    colorlinks,
    citecolor=black,
    filecolor=black,
    linkcolor=black,
    urlcolor=black
}

\usepackage{tikz-cd}
\tikzcdset{arrow style=math font}
\usetikzlibrary{external,arrows.meta}

\numberwithin{equation}{section}
\newtheorem{theorem}{Theorem}[section]
\newtheorem*{theorem*}{Theorem}
\newtheorem*{obrA}{Theorem A}
\newtheorem*{obrB}{Theorem B}
\newtheorem*{obrD}{Theorem D}

\newtheorem*{corollary*}{Corollary}
\newtheorem*{conjecture*}{Conjecture}
\newtheorem{lemma}[theorem]{Lemma}
\newtheorem{proposition}[theorem]{Proposition}
\theoremstyle{definition}
\newtheorem{question}{Question}

\newtheorem{definition}[theorem]{Definition}
\newtheorem{remark}[theorem]{Remark}

\DeclareMathOperator{\id}{id}
\DeclareMathOperator{\Aut}{Aut}
\DeclareMathOperator{\Out}{Out}
\DeclareMathOperator{\Inn}{Inn}
\DeclareMathOperator{\Hom}{Hom}

\input cyracc.def

\begin{document}

\title{A non-nilpotent localization of a finite $p$-group} 

\author[S. O. Ivanov]{Sergei O. Ivanov} 
\email{ivanov.s.o.1986@gmail.com, ivanov.s.o.1986@bimsa.cn}
\address{
Beijing Key Laboratory of Topological Statistics and Applications for Complex Systems, Beijing Institute of Mathematical Sciences and Applications (BIMSA), Beijing 101408, China.}

\author[G. Kadantsev]{Georgii Kadantsev}
\email{kadantsev.georg@yandex.ru}
\address{Qiuzhen College, Tsinghua University}

\begin{abstract}
We give negative answers to two questions of Farjoun about localizations of groups. For every odd prime $p$, we construct a finite $p$-group $G$, an infinite non-nilpotent group $H$, and a local embedding $f\colon G\hookrightarrow H$. The associated localization satisfies $L_f(G)\cong H$, showing that localizations need not preserve nilpotence and that the localization map $G\to L(G)$ need not be surjective when $G$ is a finite $p$-group.   
\end{abstract}

\maketitle

\tableofcontents

\section{Introduction}
A localization on the category of groups is an endofunctor
$L\colon \mathrm{Grp}\to\mathrm{Grp}$ equipped with a natural
transformation $\eta\colon \mathrm{Id}\to L$ such that
$
\eta_{LG},L\eta_G\colon LG\to L^2G$ are equal isomorphisms for every group $G$. For example: 
$n$-nilpotentization $G\mapsto G/\gamma_{n+1}(G)$,
Baumslag's rationalization $G\mapsto G\otimes\mathbb{Q}$,
Bousfield's $HR$-localization, and Levine's localization are localizations. 
For an overview of the subject and its history, see
\cite{Casacuberta,Libman,AIP}.
A central theme is determining which classes of groups are preserved
by localization functors.

Many classes of groups are known not to be preserved by a general localization:
\begin{itemize}
    \item finite groups \cite{LibmanFinite}; 
    \item torsion groups \cite{Libman};
    \item simple groups \cite{RodriguezSchererViruelSimple};
    \item perfect groups \cite{RodriguezSchererViruel,BadziochFeshbach};
    \item  solvable groups \cite{OSullivanFree,OSullivan}.
\end{itemize}
Every localization does, however, preserve:
 abelian groups
\cite{Casacuberta}; and  
nilpotent groups of class at most 2
\cite{Libman}.

The following two questions, posed by Farjoun in the 1990s, remained open and were regarded as among the most important problems in the theory \cite{Libman1997,Casacuberta,Farjoun2006}.
\begin{question}
Does every localization preserve the class of nilpotent groups?
\end{question}
\begin{question}
For every localization $L$ and every finite $p$-group $G$,
is the map $G\to LG$ surjective?
\end{question}

Since a finite group is nilpotent if and only if it is the product of its Sylow $p$-subgroups, and since any localization preserves finite products, the second question can be viewed as a refinement of the first for the case of finite groups. Both questions have affirmative answers for certain classes of
localizations: those that preserve reflexive coequalizers
\cite{AIP}; and $R$-localizations when $R$ is a binomial ring
\cite{KKI}. 

This paper is devoted to a negative answer to both of these questions in the general case.

Every group homomorphism $f:G\to H$ determines a localization $L_f$. A group $Y$ is called $f$-local, if  
$f^\ast\colon \Hom(H,Y)\to\Hom(G,Y)$
is a bijection. For any group $X$ there exists a universal homomorphism $X\to L_f X$ to an $f$-local group. This defines a localization functor \[L_f:\mathrm{Grp} \longrightarrow \mathrm{Grp}\] 
that is called $f$-localization. 

A homomorphism $f:G\to H$ is called local, if the group $H$ is $f$-local. It is known that in this case $L_fG\cong H.$ Notice that for any localization $L$ and any group $X$ the homomorphism $X\to LX$ is local. 

Questions 1 and 2 can therefore be reformulated as follows, for a local homomorphism $f\colon G\to H$. 
\par\medskip
\noindent\textbf{Question $1'$.}
If $G$ is nilpotent, must $H$ also be nilpotent?

\par\medskip
\noindent\textbf{Question $2'$.}
If $G$ is a finite $p$-group, must $f$ be surjective?
\par\medskip

There are several positive results concerning Question~$2'$: 
\begin{itemize}
    \item If $G$ is a finite $p$-group of nilpotency class at most $3$ and $H$ is finite, then $f$ is surjective.  \cite{Aschbacher}.
    \item If $G$ is a finite metabelian $p$-group and $H$ is finite,
    then $f$ is surjective.  \cite{Aschbacher}.
    \item If $G$ and $H$ are abelian,
    then the localization map $X\to L_fX$ is surjective for every
    finite $p$-group $X$ \cite[Theorem~B, p.~27]{BastardasThesis}.
\end{itemize}

Our main result is the following theorem. 

\begin{theorem*}
    For every odd prime $p$, there is a group $G$ of order $p^6$ and nilpotency class 3, with a local embedding $f \colon G \hookrightarrow H$ into an infinite non-nilpotent group. 
    Hence $L_f$ takes the nilpotent finite $p$-group $G$ to the infinite non-nilpotent group $H$.
\end{theorem*}

The group $G$ is an explicit metacyclic quotient of $\mathbb{Z}/{p^3} \rtimes \mathbb{Z}/{p^4}$ with $G' \cong \mathbb{Z}/p^2$ (Section~\ref{section:finite-p-group}). 
Its endomorphisms satisfy a useful dichotomy: each either kills $G'$ or is an automorphism.
The group $H$ is assembled in three steps.
\begin{enumerate}
    \item We construct a torsion-free ``Tarski-type'' monster group $S$, simple and with all proper subgroups cyclic, whose automorphism group is $S \rtimes \mathbb{Z}/p$ with $\mathbb{Z}/p$ acting regularly;
    it comes from Obraztsov's embedding theorems \cite{Obraztsov}, which rest on Ol'shanskii's graded aspherical presentations \cite{Olshanskii} (Section~\ref{section:tarski-monster-group}).
    \item 
    We build a perfect central extension $K$ of $S$ with center $\mathbb{Z} / p^2$,
    using the description of the relation module of a graded aspherical presentation over a free product \cite[\S\S 31--34]{Olshanskii}, \cite{AshmanovOlshanskii} (Section~\ref{section:the-central-extension}).
    \item The group $H$ is the quotient $(K \rtimes G)/\Delta$ that identifies the center of $K$ with $G'$ (Section~\ref{section:target}).
\end{enumerate}  

These methods are not entirely new. 
The construction follows the pattern of \cite{GobelRodriguezShelah,GobelShelah}: a rigid simple target produced by combinatorial group theory, here Obraztsov's embedding theorems. 
To it we add the passage to central extensions used in \cite{RodriguezSchererViruelSimple} and a gluing along a subgroup as in \cite{BadziochFeshbach}. 
The new ingredient, forced by the nilpotency of $G$, is an equivariant central extension whose center is identified with $G'$.

The appendix recalls the background on graded presentations and van Kampen diagrams, 
states the parts of Obraztsov's theorems that we use, and proves Lemma~\ref{lemma:counting-relators},
an analogue over free products of Ol'shanskii's description of the relation module of a graded aspherical presentation.

Some questions remain open. 
\begin{question}
    
Let $f\colon G \to H$ be a local homomorphism between finite groups. 
If $G$ is nilpotent, must $H$ be nilpotent?
\end{question} 
\begin{question}
If $G$ is a finite $2$-group, is it true that $LG$ is nilpotent for every localization $L$?
\end{question}
Our construction does not extend to $2$-groups: $G$ must act non-trivially on the monster $S$, so the image of a $2$-group $G$ in $\Aut(S)$ would contain an involution, whereas in Obraztsov's theorems the factors have no involutions, and every element of finite order of $\Aut(S)$ is conjugate into a factor (Lemma~\ref{lemma:conjugate-into}).

\subsection*{Conventions} 
We write $x^y = y^{-1}xy$, $[x,y] = x^{-1}y^{-1}xy = x^{-1}x^y$ and $c_y(x) = x^y$.
The lower central series is $\gamma_1(G) = G$, $\gamma_{i+1}(G) = [\gamma_i(G), G]$, so $\gamma_2(G) = G' = [G,G]$. 
We write $C_n$ for the cyclic group of order $n$ in multiplicative notation, and $\mathbb{Z}/n$ for the integers modulo $n$ as an additive group or a ring; in the group ring $\mathbb{Z}[C_p]$ the generator of $C_p$ is denoted $t$.

\section{The finite \texorpdfstring{$p$}{p}-group \texorpdfstring{$G$}{G}}\label{section:finite-p-group}

Fix an odd prime $p$ and consider the group
\[ F = C_{p^3} \rtimes C_{p^4} = \langle a \rangle \rtimes \langle b \rangle,\]
with the action given by $a^b = a^{1+p}$.
This is well defined since $a^{b^k} = a^{(1+p)^k}$ and $1+p$ has order $p^2$ in $(\mathbb{Z}/p^3)^\times$.
The element $x = b^{p^3}a^{-p^2}$ is central of order $p$.
Indeed, $a^{p^2}$ is central, since $(a^{p^2})^b = a^{p^2(1+p)} = a^{p^2}$, and $b^{p^3}$ is central because already $a^{b^{p^2}} = a^{(1+p)^{p^2}} = a$;
hence $x^p = b^{p^4}a^{-p^3} = 1$.

The subgroup $\langle x\rangle$ meets $\langle a\rangle$ and $\langle b\rangle$ trivially,
so in the quotient by $\langle x \rangle$ the images of $a$ and $b$ retain their order.
Set $G := F/\langle x\rangle$, a group of order $p^6$ with presentation
\begin{equation}\label{eq:fin-rep}
    \langle a, b \mid a^{p^3} = b^{p^4} = 1, a^{p^2} = b^{p^3}, a^b = a^{1+p} \rangle.
\end{equation}
Every element of $G$ can be written uniquely as $a^ib^j$ with $0 \le i,j < p^3$.

We have $[a,b] = a^{-1}a^b = a^p =: q$ and
\begin{align}\label{eq:action-on-comm}
    q^b = (a^b)^p = q^{1+p},
\end{align}
so $\langle q\rangle$ is normal and contained in $[G,G]$.
Since $[a,b] = q$, the quotient $G/\langle q\rangle$ is abelian, so $[G,G] = \langle q\rangle \cong C_{p^2}$.
The center can be read off from the normal form:
\begin{align}
    a^{a^i b^j} = a^{b^j} = a^{(1+p)^j}, \\
    (a^i b^j)^{b} = a^{(1+p)i} b^j,
\end{align}
so $a$ commutes with $a^i b^j$ if and only if $(1+p)^j \equiv 1 \pmod{p^3} \iff p^2 \mid j$.
Similarly, $b$ commutes with $a^i b^j$ if and only if $(1+p)i \equiv i \pmod{p^3} \iff p^2 \mid i$.
Therefore, $Z(G) = \langle a^{p^2}, b^{p^2} \rangle = \langle b^{p^2} \rangle$, of order $p^2$, since $b^{p^3} = a^{p^2}$.
Finally, $[q,b] = a^{-p}(a^b)^p = a^{p^2}$, so $\gamma_3(G) = \langle a^{p^2}\rangle$ is central and $\gamma_4(G) = 1$.

\begin{lemma}
    For the group $G$ given by the presentation \eqref{eq:fin-rep} the following assertions hold:
    \begin{enumerate}\itemsep=0mm
        \item it is a finite group of order $p^6$;
        \item its derived subgroup $G'$ is cyclic of order $p^2$;
        \item its center is cyclic of order $p^2$;
        \item it is nilpotent of class 3.
    \end{enumerate}
    Furthermore, the centralizer $D := C(G') = \langle a, b^p \rangle$ has index $p$.
    This normal subgroup is the kernel of the map $\varepsilon: G \to \mathbb{Z}/p,\ a \mapsto 0, b \mapsto 1$.
\end{lemma}
\begin{proof}
    Only the last assertion remains.
    Clearly $a \in D$, while $b$ acts on $G'$ by multiplication by $1+p$ modulo $p^2$.
    Hence $a^ib^j \in D$ if and only if $p \mid j$, so $D = \langle a, b^p\rangle$ has index $p$.
\end{proof}

Recall the Hall--Petresco formula \cite{Hall} for class $3$ groups:
\[ x^n y^n = (xy)^n \tau_2^{\binom{n}{2}} \tau_3^{\binom{n}{3}}, \]
where $\tau_2 \in \gamma_2(G), \tau_3 \in \gamma_3(G)$.

If $\exp(\gamma_i(G))$ divides $\binom{n}{i}$ for $i = 2,3$, the correction terms vanish and $x^ny^n = (xy)^n$.
Here $\exp(\gamma_2(G)) = p^2$ and $\exp(\gamma_3(G)) = p$, so $n = p^3$ works and $\xi(x) = x^{p^3}$ is an endomorphism of $G$.
Indeed, for every odd prime $p$, we have $p^2 \mid \binom{p^3}{2}$ and $p \mid \binom{p^3}{3}$.
It is clear that $\xi$ maps into $Z(G)$ and has kernel exactly $C(G')$.
\begin{remark}
    This also works for $n = p^2$, since $p^2 \mid \binom{p^2}{2}$ and $p \mid \binom{p^2}{3}$:
    the map $x \mapsto x^{p^2}$ is an endomorphism of $G$ onto $Z(G)$.
\end{remark}
The same argument applies to $D = \langle a, b^p\rangle$.
Here $[a,b^p] = a^{-1}a^{(1+p)^p} = a^{p^2}$, since $(1+p)^p \equiv 1+p^2 \pmod{p^3}$ for odd $p$.
This commutator is central, so $D$ has class $2$ and $D' = \langle a^{p^2}\rangle \cong C_p$.
For $n = p$ only the term $\tau_2^{\binom{p}{2}}$ remains, and it is trivial because $p \mid \binom{p}{2}$, so $x \mapsto x^p$ is an endomorphism of $D$.
We obtain:
\begin{lemma}
    In $G$ the following equations hold:
    \begin{equation}\label{eq:3-5}
        (a^i b^j)^{p^3} = a^{p^2j}
    \end{equation}
    \begin{equation}\label{eq:3-6}
        (a^i b^{pk})^p = a^{pi} b^{p^2k}
    \end{equation}
\end{lemma}

We also need the Frattini subgroup.
\begin{lemma}
    The Frattini subgroup is $\Phi(G) = \langle a^p, b^p \rangle = \{ a^{pi} b^{pj}\}$, a group of order $p^4$.
\end{lemma}
\begin{proof}
    For any finite $p$-group $X$ the Frattini subgroup is $\Phi(X) = X^p [X,X]$ (\cite[23.2, p. 105]{AschbacherFGT}).
    Equivalently, $N = \langle a^p, b^p\rangle$ is the kernel of $\pi \colon G \to \mathbb{F}_p^2,\ a^i b^j \mapsto (i \bmod p, j \bmod p)$.
\end{proof}

We can now describe the endomorphisms of $G$.
\begin{proposition}\label{prop:p-group-endomorphisms}
    For every endomorphism $\alpha : G \to G$ one of the following holds:
    \begin{enumerate}
        \item $\alpha$ is an automorphism, $\xi \circ \alpha = \xi$ and $\alpha(q) = q^{(1+p)^d}$ for a unique $d \in \{ 0, 1, \ldots, p-1\}$;
        \item $\alpha(G') = 1$.
    \end{enumerate}
\end{proposition}

\begin{proof}
    Write $\alpha(a) = a^i b^j$ and $\alpha(b) = a^lb^m$.
    Since $\alpha(a)^{p^3} = 1$, \eqref{eq:3-5} gives $p \mid j$, say $j = pk$.
    Now $\alpha(a)^p = \alpha(q) \in G' = \langle q\rangle$, while $\alpha(a)^p = a^{pi}b^{p^2k}$ by \eqref{eq:3-6};
    by uniqueness of the normal form $p \mid k$, so $\alpha(a) = a^ic^s$ with $c = b^{p^2}$.

    Applying $\alpha$ to $a^{p^2} = b^{p^3}$ gives $a^{p^2i} = \alpha(a)^{p^2} = \alpha(b)^{p^3} = b^{p^3m} = a^{p^2m}$,
    so $i \equiv m \pmod p$.

    On the one hand, applying $\alpha$ to $[a,b] = a^p$ gives
    \begin{align}
        [a^i c^s, a^l b^m] &= [a^i, a^l b^m] = a^{-i} (a^l b^m)^{-1} a^i a^l b^m \notag\\
        &= a^{-i} b^{-m} a^i b^m = a^{i((1+p)^m - 1)}.
    \end{align}
    On the other hand,
    \begin{align}\label{eq:a-cubed}
        \alpha(a)^p = a^{pi}c^{ps} = a^{pi + p^2s}
    \end{align}
    Hence,
    \begin{align}\label{eq:second-congruence}
    i((1+p)^m-1) \equiv pi + p^2s \pmod {p^3}.
    \end{align}
    Dividing by $p$ and reducing modulo $p$, using $\frac{(1+p)^m-1}{p} \equiv m \pmod p$, we get $i(m-1) \equiv 0 \pmod p$.
    As $m \equiv i \pmod p$, this gives $i(i-1) \equiv 0 \pmod p$, so $i \equiv 0$ or $i \equiv 1 \pmod p$.

    \noindent \textit{Case $i \equiv 1 \pmod{p}$}.
    Then $m \equiv 1 \pmod{p}$.
    Consider the induced map $\bar{\alpha} : \mathbb{F}_p^2 \to \mathbb{F}_p^2$:
    \begin{align*}
        \alpha(a) = a^i c^s & \mapsto (i, 0) = i \bar{a};\\
        \alpha(b) = a^l b^m & \mapsto l \bar{a} + m\bar{b}.
    \end{align*}
    Writing the matrix of $\bar{\alpha}$ in the basis $\bar{a}, \bar{b}$, we get
    \[ M = \begin{pmatrix}
        i & l \\
        0 & m \\
    \end{pmatrix} \quad \det M = im. \]
    We have $i \equiv m \equiv 1 \pmod{p}$,
    so $\det M \ne 0$ and $\bar\alpha$ is bijective, i.e. $\alpha(G)\Phi(G) = G$.
    Since the elements of $\Phi(G)$ are non-generators, $\alpha(G) = G$, and $\alpha$ is bijective because $G$ is finite.
    By \eqref{eq:a-cubed}, $\alpha(q) = q^{i+ps}$ with $i + ps \equiv 1 \pmod p$;
    modulo $p^2$ the elements congruent to $1$ modulo $p$ are exactly the powers $(1+p)^d$, $0 \le d < p$, since $(1+p)^d \equiv 1+pd \pmod{p^2}$.
    Thus $\alpha(q) = q^{(1+p)^d}$ for a unique $d \in \{ 0, 1, \ldots, p-1 \}$.

    Finally, for all $u, v \in \mathbb{Z}$
    \begin{align*}
        \xi(\alpha(a^v b^u)) &= \alpha(a^v b^u)^{p^3} = (\alpha(a)^v \alpha(b)^u)^{p^3} \\
        &= (a^{v i} c^{v s} (a^l b^m)^u)^{p^3} = (b^{p^3})^{mu} = a^{p^2u m}.
    \end{align*}
    Using that $m \equiv 1 \pmod{p}$, we have $p^2m u \equiv p^2u \pmod{p^3}$.
    Hence $\xi\circ\alpha = \xi$.

    \noindent \textit{Case $i \equiv 0 \pmod{p}$}.
    Then $m \equiv 0 \pmod p$, so $(1+p)^m \equiv 1 \pmod{p^2}$.
    Together with $p \mid i$, we get from the congruence \eqref{eq:second-congruence} that $pi + p^2s \equiv 0 \pmod{p^3}$.
    Thus $\alpha(q) = a^{pi + p^2s} = 1$, and $\alpha$ kills $G'$.
\end{proof}

\begin{remark}\label{remark:xi-notation}
    Since $\xi(g) = a^{p^2\varepsilon(g)}$, the identity $\xi\circ\alpha = \xi$ gives
    $\varepsilon\circ\alpha = \varepsilon$ for every automorphism $\alpha$ of $G$.
\end{remark}

\section{The monster group \texorpdfstring{$S$}{S}}\label{section:tarski-monster-group}
In this section we construct an infinite simple torsion-free group $S$ all of whose proper subgroups are cyclic. 
It has an outer automorphism $\beta$ of order $p$ without non-trivial fixed points, and $\Aut(S) = \Inn(S) \rtimes \langle\beta\rangle$.

Let $F = \langle t\rangle * \langle x\rangle * \langle y\rangle$ be the free product with $t^p = 1$, 
and let $\pi\colon F \to \mathbb{Z}/p$ (written additively) be the projection sending $t$ to 1 and killing $x, y$; 
it plays the role of $\varepsilon$ from Remark~\ref{remark:xi-notation}. 
In this section $N := \ker\pi$ and $F_0 := \langle x\rangle * \langle y\rangle \subseteq N$; 
$\exp_x$ and $\exp_y$ denote exponent sums in $x$ and $y$, and relators are denoted by lowercase letters.

The following proposition collects what we need about $S$; 
its last assertion is used in Section~\ref{section:the-central-extension}
to construct a central extension $K$ of $S$ to which the action of $T$ lifts.
\begin{proposition}\label{prop:simple-monster}
    There is a graded set $\mathcal{R}\subseteq N$ with normal closure $R \subseteq F$ such that the groups $T = F/R$ and $S = N/R$ satisfy
    \begin{enumerate}\itemsep=0mm
        \item $\langle F \mid \mathcal{R} \rangle$ is graded diagrammatically aspherical;
        \item $S$ is an infinite simple torsion-free group with every proper subgroup cyclic;
        \item the conjugation action of $T$ on $S$ identifies $\Aut(S) = T$ and conjugation by $t$ is regular (in particular, $\Out(S) = C_p$);
        \item every finite subgroup of $T$ is conjugate into $\langle t \rangle$.
    \end{enumerate}
    Moreover, there are integers $n \ge 1$ and $h \ge 2$ such that $\mathcal{R}$ contains words $r_5, r_6, r_7 \in \langle x, y\rangle$ with $\exp_y(r_k) = h - 2$ and
     \begin{align}
        \exp_x(r_k) = hn + k(h-1) + 15(h-1)(h-2) 
    \end{align}
    for $k = 5, 6, 7$.
\end{proposition}
\begin{proof}
We apply Obraztsov's theorems with $G_1 = \langle x\rangle$, $G_2 = \langle y\rangle$, $G_3 = H = C_p = \langle t\rangle$, $m = \infty$, $g_1, g_2$ trivial and $g_3$ the identity. 
Since $p$ is odd, all three factors have no involutions. 
Then $G(1) = F$, the images of the $g_i$ generate $H$, and $N_1 = \langle x\rangle$, $N_2 = \langle y\rangle$ are non-trivial. 
Theorems A, B and D give a presentation $T := \langle G(1) \mid \mathcal{R}\rangle$ and a normal subgroup $S := L$ of $T$, infinite and simple, with $T/S \cong C_p$.

Let $\tau\colon F \to T$ and $\kappa\colon T \to T/S$ be the natural projections, so that $R = \ker\tau$ is the normal closure of $\mathcal{R}$, and let $\rho = \kappa\circ\tau\colon F \to T/S \cong C_p$.
By Theorem A(1), $x$ and $y$ map into $S$, so their normal closure $N$ lies in $\ker\rho$;
both have index $p$, as $\pi$ and $\rho$ are onto, so $N = \ker\rho$ and in particular $\mathcal{R} \subseteq N$. 
Moreover, $\tau^{-1}(S) = \ker\rho = N$ and $\tau$ is onto, so $S = \tau(N) = N/R$.

It remains to prove items (2)--(4) and the final assertion. 
Since all $G_i$ are cyclic, Theorem D shows that every proper subgroup of $S$ is cyclic. 
For torsion-freeness, let $X \in S$ be non-trivial. 
If $X$ is conjugate into $\langle x\rangle$ or $\langle y\rangle$, it has infinite order; 
it is not conjugate into $\langle t\rangle$, as its image in $T/S$ would be non-trivial. 
Otherwise, since $C_p$ has no elements of even order, Theorem A(2) shows that $X$ has infinite order.

Item (3) is Theorem A(3), and regularity follows from Theorem B. 
For item (4), let $A \subseteq T$ be finite. 
Since $S$ is torsion-free, $A \cap S = 1$, so $A$ embeds into $T/S \cong C_p$; 
hence $A = \langle u\rangle$ is cyclic. 
By Lemma~\ref{lemma:conjugate-into}, $u$ is conjugate to an element of $\langle x\rangle$, $\langle y\rangle$ or $\langle t\rangle$, and only the last is possible, since $\langle x\rangle$ and $\langle y\rangle$ are torsion-free. 
So $A$ is conjugate into $\langle t\rangle$.

The words $r_k$ are relators of the second type \cite[(2.2), p.~277]{Obraztsov}:
\[r_k = c^{-1}A^n c A^{n+k} c A^{n + 30 + k} \ldots c A^{n + 30(h-2) + k}\]
for suitable elements $a, c$, where $A = a(aca^{-1}c^{-1})^{\ell} = a[a^{-1},c^{-1}]^{\ell}$ with $\ell \geq 1$ is a period (Obraztsov writes commutators as $uvu^{-1}v^{-1}$, so this is his $a[a,c]^{\ell}$);
that one may take this word to be a period is \cite[p.~277]{Obraztsov}, which refers to \cite[Lemma~3.1]{Obraztsov1994}.
We can take $a = x$ and $c = y$. 
For this we well-order $I_1$ so that $\langle x\rangle$ precedes $\langle y\rangle$, and choose the ordering of $\Omega$ so that $x$ is its least element; 
then $x$ is the minimal element of $F(\{A\})$, as required before \cite[(2.2)]{Obraztsov}.
We also take $h = \delta^{-1} \ge 2$ and $n = [n_0/(h+1)]$, with $\delta$, $n_0$ as in \cite[\S2, p.~276]{Obraztsov}. 
Then $\exp_x(A) = 1$ and $\exp_y(A) = 0$, so $\exp_y(r_k) = h - 2$ and 
\[ \exp_x(r_k) = n + \sum_{j = 0}^{h-2}{(n + 30j + k)} = hn + k(h-1) + 15(h-1)(h-2). \]
These relators exist for $5 \le k \le 15$; in particular $r_5, r_6, r_7$ are available.
\end{proof}

\begin{remark}\label{remark:endomorphism-dichotomy}
    As for any simple group, every endomorphism $f\colon S \to S$ is either trivial or injective. 
    If $f$ is injective, then $f(S) \cong S$ is not cyclic, so $f(S)$ is not a proper subgroup and $f$ is an automorphism. 
    Hence every endomorphism of $S$ is either trivial or an automorphism.
\end{remark}

\section{The central extension \texorpdfstring{$K$}{K}}\label{section:the-central-extension}
We first recall some facts about central extensions and perfect groups, following \cite{milnor-k-theory}.
\begin{definition}
    A central extension of a group $Y$ is a surjective homomorphism $X \twoheadrightarrow Y$ with its kernel contained in the center.     
\end{definition}

\begin{definition}
    A group $X$ is called \emph{perfect} if it is equal to its own commutator subgroup: $X = [X, X]$. 
\end{definition}

We will use the following standard fact.
\begin{lemma}[\cite{milnor-k-theory}, Lemma 5.4]\label{lemma:uniqueness-map}
    Let $\pi: Y \to X$ be a central extension and $Z$ be a perfect group.
    If two homomorphisms $f_1, f_2 : Z \to Y$ are such that $\pi \circ f_1 = \pi \circ f_2$, then $f_1 = f_2$.
\end{lemma}
\begin{proof}
    Put $d(a) = f_1(a)^{-1} f_2(a)$. 
    By assumption $\pi \circ d = 1$, so the image of $d$ lies in the kernel of $\pi$ and is thus central.
    Hence $d$ is a homomorphism: $d(ab) = f_1(b)^{-1} d(a) f_2(b) = d(a) d(b)$ since $d(a)$ is central. 
    Since its image is abelian, $d$ factors through the abelianization of $Z$, which is trivial because $Z$ is perfect. Hence $d$ is trivial and $f_1 = f_2$.
\end{proof}

We keep the notation of Section~\ref{section:tarski-monster-group} and construct a perfect central extension $K$ of $S$. 
First we specialize Lemma~\ref{lemma:counting-relators} to our presentation.
\begin{lemma}\label{lemma:modified-relator-count}
    Choose a relator $r_0 \in \mathcal{R}$ that is not a proper power in $F$. 
    There is a homomorphism 
    \[\lambda: R/[R,R] \to \mathbb{Z}[C_p]\] satisfying
    \begin{enumerate}\itemsep=0mm
        \item $\lambda(r^f) = \lambda(r) t^{\pi(f)}$ for any $ r\in R, f\in F$;
        \item $\lambda(r_0) = 1$ and $\lambda(r) = 0$ for any $r\in \mathcal{R}$ not conjugate to $r_0^{\pm 1}$.
    \end{enumerate}
\end{lemma}
\begin{proof}
    Lemma~\ref{lemma:counting-relators} assumes that no relator is conjugate to another relator or to its inverse. 
    We may assume this for Obraztsov's presentation: within one rank we discard duplicates, and across ranks it cannot happen, by asphericity. 
    The lemma gives $\lambda_{r_0}\colon R/[R,R] \to \mathbb{Z}[T]$, and we compose it with the map $\mathbb{Z}[T] \to \mathbb{Z}[C_p]$ induced by $T \to C_p$, which exists because $R \subseteq N = \ker\pi$.
\end{proof}
The relator $r_5$ satisfies the assumption of the lemma: from \cite[\S 3]{Obraztsov} the presentation satisfies condition R of \cite[\S\S 25, 34]{Olshanskii}.
In particular, no relator is a proper power (condition R5 guarantees this for relators of the second type, which is the only type the presentation uses given the choice $m = \infty$).
Taking $r_0 = r_5$, the lemma gives the map $\lambda : R/[R,R] \to \mathbb{Z}[C_p]$.

Let $\eta\colon \mathbb{Z}[C_p] \to \mathbb{Z}/p^2$ be the ring homomorphism with $t \mapsto 1+p$ (well defined since $(1+p)^p \equiv 1 \pmod{p^2}$), and let $\chi$ be the composite
\[ \chi := \eta \circ \lambda \circ (R \to R / [R,R]): R \to \mathbb{Z}/p^2 \quad \chi(r^f) = \chi(r) (1+p)^{\pi(f)}. \]
Then $\ker\chi$ is normal in $F$ by property (1) of Lemma~\ref{lemma:modified-relator-count}, and we set $K_1 := F/\ker\chi$ and $K := N/\ker\chi$.
By abuse of notation, we identify $t$ with its image in $K_1$.

Consider the element $w = r_5 r_6^{-2} r_7 \in R$. 
Its exponent sums are $\exp_x(w) = (h-1)(5 + 7 - 2\cdot 6) = 0$ and $\exp_y(w) = 0$, so, since $F_0$ is free on $x$ and $y$, $w \in [F_0,F_0] \subseteq [N,N]$; 
let $z$ denote its image in $K$ and in $K_1$. 
The relators $r_6, r_7$ are not conjugate to $r_5^{\pm1}$, since the values of $\exp_x$ on $r_5, r_6, r_7$ are distinct and non-zero; 
so by Lemma~\ref{lemma:modified-relator-count}, $\lambda(r_5) = 1$ and $\lambda(r_6) = \lambda(r_7) = 0$, and hence $\chi(w) = 1$.

\begin{proposition}\label{prop:K-perfect-extension}
    The group $K$ is a perfect central extension of $S$. 
    Moreover, $Z(K) = \langle z \rangle \cong C_{p^2}$ and coincides with the set of torsion elements of $K$. 
\end{proposition}
\begin{proof}
    Let $A := R / \ker \chi$.
    First note that $\chi$ is onto, so by the first isomorphism theorem $A \cong C_{p^2}$.
    For $n \in N$ and $r \in R$ we have $\chi(r^n) = \chi(r)$, since $N = \ker\pi$; 
    hence $[R,N] \subseteq \ker\chi$.
    Therefore, $A$ is central in $K$ and is generated by $z$ since $\chi(w) = 1$.

    By the third isomorphism theorem, $K/A \cong N/R \cong S$, which has trivial center, so $Z(K) = A$.
    In particular, $K$ is a central extension of $S$ as there is the short exact sequence 
    \[ 1 \to Z(K) \cong C_{p^2} \to K \to S \to 1. \]
    Since $w \in [N,N]$, we have $z \in [K,K]$, so $Z(K) \subseteq [K,K]$; as $K/Z(K) \cong S$ is perfect, so is $K$.
    
    For the last claim, suppose $k \in K$ has finite order. 
    Its image in $S$ has finite order, hence is trivial because $S$ is torsion-free (Proposition~\ref{prop:simple-monster}(2)); so $k \in Z(K)$.
\end{proof}
\begin{remark}\label{remark:action-on-z}
    We have $K_1 = K \rtimes \langle t\rangle$, since $K$ is normal of index $p$ and $t \notin K$ has order $p$. 
    Moreover, $z^t = z^{1+p}$: since $w^t \in R$, $\chi(w^t) = (1+p)^{\pi(t)}\chi(w) = 1+p$ and $\chi$ induces $Z(K) \cong \mathbb{Z} / p^2$. 
    In other words, $t$ acts on $Z(K)$ as multiplication by $1+p$.
\end{remark}

Conjugation by $t$ plays a special role, and we write $\beta(k) = k^t$ for $k \in K$. 
Let $\sigma\colon K \to S$ be the projection.
\begin{proposition}\label{prop:k-endomorphisms}
    Every automorphism of $K$ is conjugation by an element of $K_1$, that is, of the form $\beta^e\circ c_k$ with $e \in \{0,1,\ldots,p-1\}$ and $k \in K$. 
    Every non-trivial endomorphism of $K$ is an automorphism.
    
    Moreover, if $\beta(k) k^{-1} \in Z(K)$, then $k \in Z(K)$. 
\end{proposition}
\begin{proof}
    Let $\psi \in \Aut(K)$.
    It has to preserve the center $Z(K)$, so it induces an automorphism $\psi'$ of $S$.
    By Proposition~\ref{prop:simple-monster}(3), $\psi'$ is conjugation by some $\tau \in T$.
    Since $\ker\chi \subseteq R$, the projection $\sigma_1\colon K_1 \to T$ is onto; let $\tau_1$ be a preimage of $\tau$.
    The projection $\sigma_1$ restricts to $\sigma$, so  
    \[ \sigma \circ c_{\tau_1} = \sigma \circ \psi.\]
    Then Lemma~\ref{lemma:uniqueness-map} says $\psi = c_{\tau_1}$.

    Now suppose $\psi$ is a non-trivial endomorphism.
    By Proposition~\ref{prop:K-perfect-extension}, $Z(K)$ is the set of elements of finite order,
    so $\psi$ maps $Z(K)$ to itself.
    Hence, $\psi$ induces an endomorphism $\psi'$ of $S$.
    By Remark~\ref{remark:endomorphism-dichotomy} there are two possibilities: $\psi'$ is an automorphism or $\psi'$ is trivial. 
    If it is trivial, then $\psi(K) \subseteq Z(K)$ is abelian and $\psi$ kills $K = [K, K]$, in other words, $\psi$ is trivial.
    If $\psi'$ is an automorphism, the argument above shows that $\psi$ is conjugation by an element of $K_1$.

    For the last claim, the assumption says the image of $k$ in $S$ is fixed by conjugation by $t$, so the image must be trivial by Proposition~\ref{prop:simple-monster}(3). 
\end{proof}

\section{The target group \texorpdfstring{$H$}{H} and locality}\label{section:target}
The target group is built in two more steps.

First, take the semidirect product 
\[ H_1 := K \rtimes G \]
where $g$ acts on $K$ by $\beta^{\varepsilon(g)}$.
We write elements of $H_1$ as $gk$ with the product rule 
\[ (g_1 k_1) (g_2 k_2) = g_1 g_2 \cdot \beta^{\varepsilon(g_2)}(k_1) k_2. \]

In particular $k^a = k$ and $k^b = \beta(k)$ for $k \in K$.
Let $\Delta = \langle qz^{-1}\rangle$. 
Since $\varepsilon(q) = 0$, $q$ centralizes $K$, in particular $z$; 
hence $(qz^{-1})^n = q^nz^{-n}$, which is $1$ if and only if $p^2 \mid n$, so $\Delta \cong C_{p^2}$.
The subgroup $\Delta$ is normal: $K$ and $a$ centralize $qz^{-1}$, 
while $q^b = q^{1+p}$ by \eqref{eq:action-on-comm} and $z^b = \beta(z) = z^{1+p}$ by Remark~\ref{remark:action-on-z}, so $(qz^{-1})^b = (qz^{-1})^{1+p}$.

Finally, take $H := H_1 / \Delta$ (in other words, we identify $q$ and $z$ in $H_1$).
As above, $\Delta \cap G = \Delta \cap K = 1$, so $G$ and $K$ embed into $H$; we identify them with their images.

\begin{lemma}\label{lemma:comm-is-perfect}
    We have $K = [H,H]$.
    In particular, $H$ is not solvable, and hence not nilpotent.
\end{lemma}
\begin{proof}
    The subgroup $K$ is the kernel of the surjection $p_1\colon H_1 \to G$, $gk \mapsto g$, which sends $qz^{-1}$ to $q$; 
    so $p_1(K\Delta) = \langle q\rangle$ and
    \[ H / K \cong (H_1 / \Delta) / (K\Delta / \Delta) \cong H_1 / K\Delta \cong G / (p_1(K\Delta)) = G / \langle q \rangle = G / G'.\] 
    Hence $[H,H] \subseteq K$, and conversely $K = [K,K] \subseteq [H,H]$ because $K$ is perfect.
\end{proof}

\begin{lemma}\label{lemma:conjugate-into-g}
    Every finite subgroup of $H$ lies in a conjugate of $G$. 
\end{lemma}
\begin{proof}
    Consider the map $\rho_1 : H_1 \to T$ given by $\rho_1(gk) = t^{\varepsilon(g)} \sigma(k)$. 
    It is a homomorphism: 
    \begin{align*}
    \rho_1(g_1 k_1 g_2 k_2) = \rho_1(g_1g_2 \beta^{\varepsilon(g_2)}(k_1) k_2) = t^{\varepsilon(g_1g_2)} \sigma(\beta^{\varepsilon(g_2)}(k_1) k_2) = \\
    = t^{\varepsilon(g_1)} t^{\varepsilon(g_2)} t^{-\varepsilon(g_2)} \sigma(k_1) t^{\varepsilon(g_2)} \sigma(k_2) 
    = \rho_1(g_1 k_1) \rho_1(g_2 k_2).  
    \end{align*}
    (We used that $K_1 \to T$ restricts to $\sigma$ and fixes $t$.) 
    Since $\varepsilon(q) = 0$ and $\sigma(z) = 1$, $\rho_1(qz^{-1}) = 1$, 
    so $\rho_1$ factors through $\rho\colon H \to T$. 
    It is surjective, since $\rho(b) = t$, $\rho(K) = S$ and $T = S\langle t\rangle$.
    
    Observe that $\rho^{-1}(\langle t \rangle) = G$.
    Indeed, if $\rho(gk) \in \langle t\rangle$ then $\sigma(k) = 1$ using that $S \cap \langle t \rangle = 1$.
    Thus $k\in Z(K) = \langle z \rangle$ and $z = q \in G$ in $H$, so $gk \in G$.

    Now suppose $A \subseteq H$ is finite. 
    Then $B := \rho(A)$ is a finite subgroup of $T$ and by Proposition~\ref{prop:simple-monster}(4)  it is conjugate into $\langle t \rangle$ by some $u \in T$, in other words $B^u \subseteq \langle t \rangle$.
    Choose $h\in H$ with $\rho(h) = u$.
    Then $A^h \subseteq (\rho^{-1}B)^h = \rho^{-1}(B^u) \subseteq \rho^{-1}\langle t \rangle \subseteq G$.
\end{proof}

We can now prove that $G \hookrightarrow H$ is local.
\begin{theorem}
Every homomorphism $G \to H$ extends to exactly one endomorphism of $H$.
\end{theorem}
\begin{proof}
    Suppose $\alpha: G \to H$ is a homomorphism.
    Its image is finite, so by Lemma~\ref{lemma:conjugate-into-g} it lies in $uGu^{-1}$ for some $u \in H$; replacing $\alpha$ by $c_u\circ\alpha$, we may assume $\alpha(G) \subseteq G$.
    
    \emph{Existence}.     
    By Proposition~\ref{prop:p-group-endomorphisms}, $\alpha$ either kills $G'$ or is an automorphism.
    If $\alpha$ is an automorphism, then $\varepsilon \circ \alpha = \varepsilon$ and $\alpha(q) = q^{(1+p)^d}$ for some $d$. 
    Put 
    \[ \theta_\alpha : H_1 \to H_1\quad \theta_\alpha(gk) = \alpha(g) \cdot \beta^d(k). \] 
    It is a homomorphism: on one side
    \begin{align*}
        \theta_\alpha(g_1 k_1 g_2 k_2) = \theta_\alpha(g_1 g_2 \beta^{\varepsilon(g_2)}(k_1) k_2) = \alpha(g_1 g_2) \beta^{d+ \varepsilon(g_2)}(k_1) \beta^{d}(k_2);
    \end{align*}
    and on the other 
    \begin{align*}
        \theta_\alpha(g_1 k_1) \theta_\alpha(g_2 k_2) = 
        \alpha(g_1) \beta^d(k_1) \alpha(g_2) \beta^d(k_2) = \\ 
        = \alpha(g_1) \alpha(g_2) \beta^{\varepsilon(\alpha(g_2))}(\beta^d(k_1)) \beta^d(k_2)
    \end{align*}
    Using that $\varepsilon \circ \alpha = \varepsilon$, we see the two expressions coincide.
    It is bijective: the inverse homomorphism is built in the same way from $\alpha^{-1}$.
    
    By Remark~\ref{remark:action-on-z},
    \[ \theta_\alpha(qz^{-1}) = \alpha(q)\beta^d(z)^{-1} = q^{(1+p)^d} z^{-(1+p)^d} = (qz^{-1})^{(1+p)^d}, \]
    so $\theta_\alpha$ maps $\Delta$ onto itself ($(1+p)^d$ is a unit modulo $p^2$) and induces an automorphism of $H$, which extends $\alpha$ because $\theta_\alpha(g) = \alpha(g)$ for $g \in G$. 
    If $\alpha$ kills $G'$, it factors as $G \to G/G' \to G$, and the composite $H \to H/K \cong G/G' \to G$ extends $\alpha$, since $g \mapsto gG' \mapsto \alpha(g)$.
    
    \emph{Uniqueness}. 
    Let $\theta$ be an endomorphism extending $\alpha$.
    Since $K = [H, H]$, $\theta(K) = \theta([H,H]) \subseteq [H,H] = K$, and $\theta|_K$ is either trivial or an automorphism by Proposition~\ref{prop:k-endomorphisms}.

    If $\alpha$ kills $G'$, then $\theta(z) = \theta(q) = \alpha(q) = 1$, so $\theta|_K$ is not injective and has to be trivial. 
    So $\theta$ is determined by its values on $G$ and $K$, which generate $H$.

    Suppose $\alpha$ is an automorphism. 
    By replacing $\theta$ by $\theta \circ \theta_\alpha^{-1}$, we may assume $\theta$ fixes $G$ pointwise. 
    Then $\theta(z) = \theta(q) = q = z$, so $\theta|_K$ is non-trivial, hence bijective. 
    Write $\theta|_K = \beta^e \circ c_k$ with $k\in K$ and $e \in \{ 0,1,\ldots,p-1\}$. 

    First, observe that the condition $z = \theta(z) = \beta^e(c_k(z)) = \beta^e(z) = z^{(1+p)^e}$ forces $e = 0$, since $1+p$ has order $p$ modulo $p^2$, so $\theta|_K = c_k$. 
    Second, since $\theta$ fixes $b$, $\theta|_K$ commutes with $\beta$:
    $\theta|_K(\beta(m)) = \theta|_K(m^b) = \theta|_K(m)^{\theta(b)} = \theta|_K(m)^b$.
    Thus $c_k\circ\beta = \beta\circ c_k = c_{\beta(k)}\circ\beta$, and since $\beta$ is surjective, $c_k = c_{\beta(k)}$, i.e. $\beta(k)k^{-1} \in Z(K)$. 
    By Proposition~\ref{prop:k-endomorphisms}, $k \in Z(K)$, so $\theta|_K = \id$.
    Since $\theta$ is the identity on $K$ and on $G$, $\theta$ has to be the identity.
\end{proof}

We have proved: 
\begin{theorem}
    For every odd prime $p$, there are a group $G$ of order $p^6$ and class 3, a countably infinite non-nilpotent group $H$ and a local embedding $f\colon G \to H$. 
    In particular, $L_fG \cong H$, so the localization $L_f$ does not preserve nilpotency.
\end{theorem}

\appendix
\section{Group presentations}
This appendix collects the background on group presentations needed to state the results used to construct $S$. 
We follow \cite{Olshanskii} closely; the main tool is van Kampen diagrams, cell decompositions of a surface carrying additional data.

\subsection{Van Kampen diagrams}\label{subsec:vankampen}
We begin with a brief, informal introduction. 
Let $X$ be a surface.
An $n$-angular cell $(\Pi, f)$ is described by a nice enough continuous map $f : P \to X$, where $P$ is an $n$-gon in the plane and $\Pi = f(P)$.
A cell decomposition of $X$ is a finite collection $\{ (\Pi_i, f_i)\}_{i=1}^m$ with $\{ \Pi_i \}$ covering $X$ and satisfying several properties (similar to the standard definition of a CW-decomposition).

Fix an alphabet $\mathcal{A}$, write $\mathcal{A}^1 := \mathcal{A} \cup \mathcal{A}^{-1} \cup \{1\}$, 
let $F$ be the free group on $\mathcal{A}$, and let $G$ be given by the presentation
\begin{align}\label{eq:rep}
G := \langle \mathcal{A} \mid \mathcal{R} \rangle.
\end{align}

\begin{definition}
A cell decomposition $\Delta$ of an oriented surface $X$ is called a \emph{map} on $X$.
The orientation on $X$ defines an orientation of $\Delta$: we consider edges of $\Delta$ to be oriented; thus, with every edge $e$ of $\Delta$ there is the associated edge $e^{-1}$ with the opposite orientation. 

If $X$ is a disc (a sphere), $\Delta$ is called \emph{circular} (\emph{spherical}).
The boundary of a circular $\Delta$ is always traversed clockwise.
\end{definition}
\begin{definition}
If a component of the boundary consists of edges $e_1\ldots e_n$ listed in accordance with the orientation, then the path $e_1\ldots e_n$ is a loop and it is called the \emph{contour} of $\Delta$.
Similarly, we define the contour of any cell $\Pi$, denoted $\partial \Pi$.
We consider contours up to a cyclic shift. 
A circular $\Delta$ has a single contour.
\end{definition}

\begin{definition}
A map $\Delta$ equipped with an assignment $e \mapsto \phi(e) \in \mathcal{A}^{1}$ is called a \emph{diagram over} $\mathcal{A}$ if $\phi(e^{-1}) = \phi(e)^{-1}$ for any edge $e$. 
\end{definition}
\begin{definition}
Consider a path $p = e_1 \ldots e_n$ in a diagram $\Delta$ over $\mathcal{A}$. 
Define its \emph{label} as $\phi(p) := \phi(e_1) \ldots \phi(e_n)$. 
If $n = 0$, by convention $\phi(p) = 1$.  
The label of a contour of a cell or a diagram is defined up to cyclic shift, in other words, it is a cyclic word. 
Two words are \emph{visually equal}, written $\equiv$, if they coincide letter by letter.
\end{definition}

\begin{definition}
    A cell of a diagram $\Delta$ is called an $\mathcal{R}$-cell if the label $\phi(p)$ of its contour is visually equal (up to cyclic conjugacy) to some word $R \in \mathcal{R}$ or $R^{-1}$, or to some word obtained from $R, R^{-1}$ by insertions of the symbol 1.
\end{definition}

\begin{definition}\label{def:0-cells}
    A cell $\Pi$ is called a 0-cell if either the label of its contour $e_1\ldots e_n$ is visually equal to $\phi(e_1) \ldots \phi(e_n)$ with $\phi(e_i) \equiv 1$, or, for some $i \neq j$, $\phi(e_i) \equiv a, \phi(e_j) \equiv a^{-1}, \phi(e_k) \equiv 1$ for any $k \neq i, j$. 
    Such a label is $1$ in the free group.
    Edges labelled 1 are called \emph{0-edges}; the others are called \emph{$\mathcal{A}$-edges}.
\end{definition}

\begin{definition}
    A diagram on a surface $X$ over a presentation \eqref{eq:rep} is any diagram over the alphabet $\mathcal{A}$ whose cells are $\mathcal{R}$-cells or 0-cells.   
\end{definition}
\begin{lemma}[van Kampen Lemma]
    Let $W$ be an arbitrary non-empty word in the alphabet $\mathcal{A}^1$.
    Then $W = 1$ in the group $G$ with presentation \eqref{eq:rep} if and only if there exists a circular diagram over \eqref{eq:rep} whose contour label is visually equal to $W$.
\end{lemma}

\emph{The idea of the proof:} in the ``only if'' direction, we take an expression of $W$ in the free group as a member of the normal closure of $\mathcal{R}$ and construct a diagram with the contour visually equal to that expression. 
This diagram is then transformed into the required diagram by means of pasting 0-cells using a number of transformations, one of which is cancellation of adjacent mutually inverse letters.

We need to be able to cancel not only letters on edges, but also to cancel cells with each other. 
The idea is that in certain cases we can replace two ``cancellable'' cells by pasting in a circular diagram without $\mathcal{R}$-cells, thus reducing the number of $\mathcal{R}$-cells in the diagram by two.  
However, two ``cancellable'' cells need not form a subdecomposition of $\Delta$; they can still be cancelled after inserting some 0-cells, a procedure called \emph{0-refinement}.
One important particular case is \emph{0-bordering}: an $n$-cell $\Pi$ is replaced by a cell $\Pi_0$
with the same labels, called the \emph{copy} of $\Pi$, surrounded by $n$ quadrilateral 0-cells
whose new edges are labelled $1$.

\subsection{Graded diagrams and asphericity}\label{subsec:graded}
Graded presentations arise naturally when relators are defined recursively, rank by rank.
\begin{definition}
    Let $\mathcal{R}$ be a set of words in a group alphabet $\mathcal{A} \cup \mathcal{A}^{-1}$ and suppose it is decomposed into subsets $\mathcal{R} = \bigcup_{i=1}^{\infty} \mathcal{R}_i$ in such a way that a word in $\mathcal{R}_i$ cannot coincide with a cyclic conjugate of a word in $\mathcal{R}_j$ or its inverse if $i \neq j$.
    Then the presentation $\langle \mathcal{A} \mid \mathcal{R} = \bigcup \mathcal{R}_i \rangle$ is called a \emph{graded} presentation, the members of $\mathcal{R}_i$ are called \emph{relators of rank} $i$.
\end{definition}

\begin{definition}
    If two words $X$ and $Y$ are equal in the group 
    \[ G(i) = \left\langle \mathcal{A} \mid R = 1, R \in \bigcup_{j = 1}^{i} \mathcal{R}_j \right\rangle, \]
    then we say $X$ and $Y$ are equal in rank $i$ and write $X \overset{i}{=} Y$.
\end{definition}

The previous definitions adapt to the graded setting.
\begin{definition}
    A map $\Delta$ is \emph{graded} if it is equipped with an assignment of rank to each of its cells. 
    The rank of a cell $\Pi$ is denoted $r(\Pi)$. 

    A diagram is graded if its underlying map is graded. 
    There is a natural grading for a diagram over a graded presentation:
    an $\mathcal{R}$-cell whose contour label is obtained from a relator $R \in \mathcal{R}_i$ as in the definition of an $\mathcal{R}$-cell gets rank $r(\Pi) = i$, which is well defined by the definition of a graded presentation, and 0-cells get rank $0$.
\end{definition}

\begin{definition}
    For $j \geq 1$ a pair of cells $\Pi_1, \Pi_2$ of rank $j$ is said to be cancellable in a diagram $\Delta$, or a \emph{$j$-pair}, if there is a 0-refinement $\Delta'$ containing copies $\Pi_1', \Pi_2'$ of $\Pi_1, \Pi_2$ and 
    a path $t$ without self-intersections joining a vertex $O_1$ of $\Pi_1'$ to a vertex $O_2$ of $\Pi_2'$ such that $\phi(t) \overset{j-1}{=} 1$ and $\phi(p_1) = \phi(p_2)^{-1}$ in $F$, where $p_1$ and $p_2$ are the contours of $\Pi_1'$ and $\Pi_2'$ starting at $O_1$ and $O_2$ (both traversed clockwise).
\end{definition}

\begin{definition}
    A graded diagram is said to be reduced if it does not contain a $j$-pair for any $j \geq 1$.
\end{definition}

\begin{definition}
    The \emph{rank} of a diagram is the largest rank of its cells.
    Its \emph{type} $\tau = (\tau_1, \tau_2, \ldots )$ is a sequence where $\tau_i$ is the number of cells of rank $i$ in the diagram.    
\end{definition}
Types are ordered lexicographically, starting from the highest rank. Cancelling a $j$-pair lowers the type, so every graded diagram can be brought to reduced form.

\begin{definition}
    A graded presentation of $G$ is called \emph{diagrammatically aspherical} or simply \emph{aspherical} if every graded spherical diagram over it is either non-reduced or consists entirely of 0-cells.
\end{definition}

\subsection{Presentations in free products}
The presentation used to construct $S$ is over a free product of groups rather than a free group; the definitions above carry over with small changes \cite[Chapter~11]{Olshanskii}.

Let $F = *_{i\in I} G_i$ be a free product; the $G_i$ are its \emph{factors}.
The group $F$ has a normal form: any element can be uniquely expressed as a product $w_1 \ldots w_k$, in which $w_{i} \in G_{\mu(i)}, w_{i} \neq 1$ and $\mu(i) \neq \mu(i+1)$. 
It is natural to consider elements of $F$ as words in the alphabet $\mathcal{A}$, consisting of all nontrivial elements of the factors $G_i$.

We will assume that in the presentation 
\[ G = \langle F \mid \mathcal{R} \rangle \]
no relator is conjugate into a factor $G_i$, and no relator is conjugate to another relator or to the inverse of another relator. 
Diagrams are defined as before, with two changes \cite[\S\S 33.2--33.3]{Olshanskii}. 
First, the inverse of a letter of $\mathcal{A}$ is again a letter of $\mathcal{A}$. 
Second, there are more 0-cells: a cell is a \emph{0-cell} if all letters of its contour label other than $1$ lie in a single factor $G_i$ and their product is $1$ in $G_i$. 
(The 0-cells of Definition~\ref{def:0-cells} are the case of the letters $a, a^{-1}$.) 
For example, in $\langle t\rangle \cong C_p$ the word consisting of $p$ copies of $t$ labels a 0-cell. 
In particular, a 0-refinement may subdivide an edge labelled $g \in G_i$ into two edges labelled $g_1, g_2 \in G_i$ with $g_1 g_2 = g$. 
With these 0-cells the van Kampen Lemma holds over $F$.

\begin{lemma}\label{lemma:counting-relators}
  Let $G = \langle F \mid \mathcal{R} \rangle$ be a graded aspherical presentation over a free product $F$ of groups without involutions, with kernel $N$, satisfying the conventions above.
  Regard the relation module $N_{\mathrm{ab}} = N/[N,N]$ as a right $\mathbb{Z}G$-module via $(u[N,N])\cdot\bar f = u^f[N,N]$.
  Let $R_0 \in \mathcal{R}$ be a relator that is not a proper power in $F$.
  Then $R_0[N,N]$ generates a free $\mathbb{Z}G$-submodule of rank one of $N_{\mathrm{ab}}$, and this submodule is a direct summand of $N_{\mathrm{ab}}$.
  More precisely, there is a unique homomorphism $\lambda : N \to \mathbb{Z}G$ such that
  \begin{enumerate}\itemsep=0mm
    \item $\lambda(u^f) = \lambda(u)\,\bar f$ for all $u \in N$, $f \in F$;
    \item $\lambda(R_0) = 1$;
    \item $\lambda(r) = 0$ for every $r \in \mathcal{R} \setminus \{R_0\}$;
  \end{enumerate}
  it induces a $\mathbb{Z}G$-linear map $\bar\lambda\colon N_{\mathrm{ab}} \to \mathbb{Z}G$ which splits $x \mapsto R_0[N,N]\cdot x$.
  In particular, if no relator in $\mathcal{R}$ is a proper power in $F$, then $N_{\mathrm{ab}}$ is a free $\mathbb{Z}G$-module with basis $\{R[N,N] : R \in \mathcal{R}\}$.
\end{lemma}

\begin{proof}
  We follow the counting argument of \cite[Lemma~32.2]{Olshanskii}, but assign each cell its position in $G$ instead of grouping cells into equivalence classes.

  \emph{Vertex labels.}
  In a disc or spherical diagram with a fixed base vertex, label each vertex $v$ by $b(v) = \overline{\phi(g)} \in G$, where $g$ is any path from the base to $v$.
  This does not depend on $g$: the label of a closed path is trivial in $G$, because every cell label is.

  \emph{Weights.}
  A cell gets weight $0$ unless its contour label $W$, read from some vertex $v$, satisfies $W = c R_0^{\varepsilon} c^{-1}$ in $F$ with $\varepsilon = \pm 1$.
  In that case its weight is
  \[ \omega = \varepsilon\, \bar c^{-1}\, b(v)^{-1} \in \mathbb{Z}G . \]
  The sign is unique, because $R_0$ is not conjugate to $R_0^{-1}$ in $F$.
  Indeed, suppose $g^{-1}R_0g = R_0^{-1}$ for some $g \in F$; then $g^2$ centralizes $R_0$.
  The centralizer of $R_0$ is $\langle R_0 \rangle$ \cite[Theorem~33.4]{Olshanskii}, since $R_0$ has infinite order, is not conjugate into a factor and is not a proper power.
  Hence $g^2 = R_0^k$ for some $k$, and conjugating by $g$ gives $R_0^k = g^{-1}g^2g = (g^{-1}R_0g)^k = R_0^{-k}$.
  Since $R_0$ has infinite order, $k = 0$, so $g^2 = 1$.
  Then $g = 1$, because $F$ has no involutions (every element of finite order in $F$ is conjugate into a factor), and $R_0 = R_0^{-1}$, which is impossible.

  The weight does not depend on the choices.
  Another conjugator differs from $c$ by an element of $C_F(R_0) = \langle R_0 \rangle \subseteq N$, so $\bar c$ is unchanged.
  Moving the starting vertex along a boundary segment $s$ replaces $c$ by $s^{-1}c$ and $b(v)$ by $b(v)\bar s$, so $\omega$ is unchanged.
  $0$-refinements do not change any weight.

  \emph{Collars.} If two words $W_1, W_2$ are equal in $F$, the van Kampen Lemma for the presentation of $F$ without relators gives a disc of 0-cells with contour label $W_1\,1\,W_2^{-1}\,1$. Identifying its two edges labelled $1$, as in the proof of \cite[Lemma~11.2]{Olshanskii}, gives an annulus of 0-cells, a \emph{collar}: its outer contour has label $W_1$, its inner boundary read clockwise has label $W_2$, and a 0-edge joins their starting vertices. Unlike a 0-refinement, attaching a collar changes the contour label, but only to a word equal in $F$. It adds only cells of weight $0$, and if the base vertex is joined to the new one by 0-edges, it changes no vertex label.

  \emph{Spherical diagrams have total weight $0$.}
  We argue by induction on the type of the diagram, as in \cite[\S 13.2]{Olshanskii}.
  If there are no cells of positive rank, the total weight is $0$.
  Otherwise asphericity provides a $j$-pair.
  Its connecting path is trivial in rank $j-1$, hence in $G$, so the starting vertices of the two cells have the same label.
  Read from $O_1$ and $O_2$, their contour labels are mutually inverse in $F$, so if one is $cR_0^{\varepsilon}c^{-1}$ the other is $cR_0^{-\varepsilon}c^{-1}$, and their weights are opposite (or both $0$).
  
  Choose the base vertex outside the part of the diagram that is changed. After attaching collars, we may assume that the two based contour labels are visually $P$ and $P^{-1}$. Cancelling the pair replaces a disc with contour label $P\,U\,P^{-1}U^{-1}$, where $U = \phi(t) \overset{j-1}{=} 1$, by a disc of rank $<j$. The removed pair has total weight $0$.  
  Take for the new disc two mirror copies of a rank-$(j-1)$ diagram for $U = 1$, joined by a strip labelled $P$ \cite[Fig.~87]{Olshanskii}.
  Since $\overline{P} = 1$, corresponding cells of the two copies have the same vertex labels and opposite orientations, so their weights cancel.
  The total weight is therefore unchanged, while the type decreases.

  \emph{Definition of $\lambda$.}
  For $u \in N$, let $\lambda(u)$ be the total weight of a disc diagram whose contour label represents $u$ (based at the start of the contour).
  Given two such discs, attach a collar to one of them so that the two contour labels become visually equal with the same starting vertex; gluing them along the contour gives a sphere of total weight $\lambda_1 - \lambda_2 = 0$, so $\lambda$ is well defined.
  Placing two discs side by side shows $\lambda(u_1u_2) = \lambda(u_1) + \lambda(u_2)$, since $\bar u_1 = 1$.
  A whisker labelled $f^{-1}$ turns a disc for $u$ into a disc for $u^f$ and left-multiplies every vertex label by $\bar f^{-1}$; this gives (1).
  One-cell diagrams give (2) and (3); for (3), recall that no relator other than $R_0$ is conjugate to $R_0^{\pm1}$.
  Uniqueness holds because $N$ is generated by the conjugates of relators.

  \emph{The relation module.}
  Since the additive group of $\mathbb{Z}G$ is abelian, $\lambda$ vanishes on $[N,N]$, and by (1) the induced map $\bar\lambda\colon N_{\mathrm{ab}} \to \mathbb{Z}G$ is $\mathbb{Z}G$-linear.
  By (2), $\bar\lambda(R_0[N,N]\cdot x) = \bar\lambda(R_0[N,N])\,x = x$ for all $x \in \mathbb{Z}G$.
  Hence $x \mapsto R_0[N,N]\cdot x$ is injective, and its image is a free direct summand of rank one, with complement $\ker\bar\lambda$.
  For the last assertion, the elements $R[N,N]$, $R \in \mathcal{R}$, generate $N_{\mathrm{ab}}$, again because $N$ is generated by the conjugates of relators.
  If $\sum_{R} R[N,N]\cdot x_R = 0$, then applying $\bar\lambda$ for the relator $R_0$ and using (2) and (3) gives $x_{R_0} = 0$; this holds for every $R_0 \in \mathcal{R}$.
\end{proof}

\begin{remark}
  For graded aspherical presentations over a free group, the last assertion of Lemma~\ref{lemma:counting-relators} is \cite[Corollary~32.1]{Olshanskii}; see also \cite[Theorem~5 and Corollary~3]{AshmanovOlshanskii}.
  Lemma~\ref{lemma:counting-relators} is its analogue over free products, with the hypothesis imposed only on the relator $R_0$; this is all that Section~\ref{section:the-central-extension} needs.
  When no relator is a proper power, the statement that every spherical diagram has total weight $0$ for every choice of $R_0$ is the identity property of \cite[Definition~3.1]{Hitchman} for the pair formed by presentation complexes of $G$ and of $F$, and the freeness of $N_{\mathrm{ab}}$ then also follows from \cite[Proposition~3.4]{Hitchman}.
\end{remark}

\subsection{Obraztsov embedding theorems}
The last ingredients are Obraztsov's embedding Theorems A, B and D.

\begin{definition}
    The \emph{free amalgam} of a family of groups $\{ X_i \}_{i\in I}$ is the (set) union $\bigcup X_i$.
    We take this union to be ``disjoint'' in the sense that $X_i \cap X_j = \{ 1 \}$ whenever $i \neq j$.
    We will call a map out of an amalgam into a group an \emph{embedding} if it is injective and a homomorphism restricted to each $X_i$.
\end{definition}

From now on, $\{G_i\}_{i\in I}$ is a family of non-trivial groups without involutions.
Put $G(1) := *_{i\in I} G_i$ and let $\Omega^1$ be the free amalgam of the $G_i$, $\Omega := \Omega^1 \setminus \{ 1 \}$. 
We omit the definition of a generating map $f$; it appears in the hypothesis of Theorem A,
and our construction uses the explicit one given in the proof of Theorem D, which satisfies the condition of Theorem A. 

We state only the parts of Obraztsov's Theorems A, B and D that we use, adapted to our setting; for the full statements see \cite{Obraztsov}. 
Theorem~D is specialized to groups $H$ without involutions, and some conclusions are taken from Obraztsov's proofs rather
than from his statements.
\begin{obrA}[{\cite[Theorem~A, items (1)--(3)]{Obraztsov}}]\label{thm:A}
    Let $m$ be a sufficiently large odd number or $m = \infty$, and let $g_i\colon G_i \to H$ be homomorphisms with kernels $N_i$ whose images generate $H$. 
    Suppose that the set $I_1$ of indices $j$ with $N_j \ne 1$ has more than one element, let $\Omega^1_1$ be the free amalgam of $\{N_j\}_{j\in I_1}$, 
    and let $f\colon 2^\Omega \setminus \{\emptyset\} \to 2^\Omega$ be a generating map such that $f(C) \cap \Omega_1 \neq \emptyset$ whenever $C$ is not contained in a single factor $G_i$, where $\Omega_1 := \Omega^1_1 \setminus \{1\}$.
    
    Then $\Omega^1$ embeds into a group $G := \langle G(1) \mid \mathcal{R}\rangle $, given by a graded aspherical presentation and generated by $\Omega$, with the following properties:
    \begin{enumerate}
        \item $\Omega_1^1$ embeds in a normal simple infinite subgroup $L$ of $G$ such that $G / L \cong H$;
        \item if $X \in L$ is not conjugate to any element of one of the groups $G_i$, 
        then either $X$ is equal to a power of an element of infinite order whose image in $H$ has even order, 
        or $X$ is of order dividing $m$ (of infinite order if $m = \infty$);
        \item the action by conjugation of $G$ on $L$ identifies $\Aut L \cong G$ (so $\Out L = H$).
    \end{enumerate}
\end{obrA}

\begin{obrB}[{\cite[Theorem~B, items (1)--(2)]{Obraztsov}}]\label{thm:B}
    Let $H$ and $G_i$ ($i \in I$, $|I| > 1$) be groups without involutions, the $G_i$ non-trivial; 
    let $\Omega^1$ be the free amalgam of $H$ and the $G_i$, $f$ any generating map on $\Omega$, and $m$ as in Theorem A. 
    Let $G$ and $L$ be the groups given by Theorem A for the factors $G_i$ and $H$, with $g_H = \id$ and the $g_i$ trivial.

    Then $\Omega^1$ embeds into $G$ and $\Out L \cong H$, and every $h\in H \setminus \{1 \}$ induces a regular automorphism on $L$.
\end{obrB}

\begin{obrD}[{\cite[Theorem~D, item 2]{Obraztsov}}]\label{thm:D}
    Let $\{G_i\}_{i\in I}$, with $|I| > 1$, be a countable family of non-trivial countable groups without involutions (countable meaning finite or countably infinite). 
    Let $H$ be a countable group without involutions, and $m$ as in Theorem A.
    
    Then there is a generating map $f_\Omega$ for which the groups $G$ and $L$ of Theorem B are such that
    every proper subgroup of $L$ is either cyclic (of infinite order if $m=\infty$ and of order dividing $m$ otherwise) or contained in $\psi(G_i)$ for some $\psi \in \Aut L$ and $i \in I$.  
\end{obrD}

\begin{lemma}[\cite{Olshanskii}, Lemma 34.7 (1), (3)]\label{lemma:conjugate-into}
    Let $G$ be the group of Theorem A with $m = \infty$. 
    Then every element of $G$ of finite order is conjugate to an element of one of the factors $G_i$.
\end{lemma}
In Ol'shanskii \cite{Olshanskii} this lemma is stated for the group $G(\infty)$; Obraztsov applies it to the presentation of Theorem A \cite[p. 285]{Obraztsov}. 

\section*{Use of artificial intelligence} 
The initial version of the example presented in this article was suggested by an AI assistant. We also used AI assistance to refine the construction, check mathematical arguments, and draft and revise the exposition. All arguments and references were independently checked by the authors, who take full responsibility for the mathematical content and references.

\bibliographystyle{alpha}
\bibliography{references}
\end{document}